\documentclass[11pt,letterpaper]{amsart}
\usepackage{cancel}
\usepackage{latexsym, amssymb, amsmath,esint}
\usepackage{soul}
\usepackage{amsfonts, graphicx}
\usepackage{graphicx,color}

\usepackage{hyperref}
\usepackage[alphabetic]{amsrefs}

\usepackage{wasysym,stmaryrd}

\usepackage{mathtools}

\mathtoolsset{showonlyrefs}
\newcommand{\be}{\begin{equation}}
\newcommand{\ee}{\end{equation}}
\newcommand{\beq}{\begin{eqnarray}}
\newcommand{\eeq}{\end{eqnarray}}

\newtheorem{thm}{Theorem}[section]
\newtheorem{conj}{Conjecture}[section]

\newtheorem{lma}{Lemma}[section]
\newtheorem{prop}{Proposition}[section]

\newtheorem{defn}{Definition}[section]

\theoremstyle{remark}
\newtheorem{rem}{Remark}[section]
\numberwithin{equation}{section}

\newtheorem{claim}{Claim}[section]

\def\be{\begin{equation}}
\def\ee{\end{equation}}
\def\bee{\begin{equation*}}
\def\eee{\end{equation*}}

\def\Ric{\text{\rm Ric}}
\def\Rm{\text{\rm Rm}}

\def\e{\varepsilon}

\def\a{{\alpha}}
\def\b{{\beta}}

\def\euc{\mathrm{euc}}

\begin{document}

\title[]
{Rigidity of positive mass theorem with fast metric decay}

\author{Jianchun Chu}
\address{School of Mathematical Sciences, Peking University, Yiheyuan Road 5, Beijing 100871, People's Republic of China}
\email{jianchunchu@math.pku.edu.cn}

 \author{Man-Chun Lee}
\address[Man-Chun Lee]{Department of Mathematics, The Chinese University of Hong Kong, Shatin, Hong Kong, China}
\email{mclee@math.cuhk.edu.hk}

\author{Jingbo Wan}
\address[Jingbo Wan]{ Jacques-Louis Lions de Sorbonne Universit\'e, 4 place Jussieu, Paris 75005,
France}
 \email{jingbo.wan@sorbonne-universite.fr}

\renewcommand{\subjclassname}{\textup{2020} Mathematics Subject Classification}

\date{\today}

\begin{abstract} 
In this work, we consider metrics on Euclidean space with nonnegative scalar curvature and rapid decay at infinity. We show that, in dimensions four and higher, any such metric is necessarily flat if its decay rate exceeds that of the Schwarzschild metric. This complements recent works by Mazurowski–Yao and You–Zhang, thereby establishing Gromov’s conjecture on the rigidity of the positive mass theorem under fast metric decay in all dimensions. Our method also extend naturally to weakly asymptotically flat manifolds.
\end{abstract}

\maketitle

\markboth{Jianchun Chu, Man-Chun Lee, Jingbo Wan}{$C^0$-Rigidity of positive mass Theorem}

\section{Introduction} \label{sec:introduction}

In Ricci geometry, the classical Bishop--Gromov volume comparison theorem asserts that if $\Ric \geq 0$, then the volume growth of geodesic balls is bounded above by that of Euclidean space. Moreover, rigidity holds in the sense that equality is attained if and only if the manifold is (locally) isometric to Euclidean space. Extending some geometric comparison principles to the setting of scalar curvature has been a central theme in Riemannian geometry. In contrast to the Ricci curvature setting, the topology of the underlying manifold plays a significant role in scalar curvature geometry. Over the past several decades, substantial progress has been made in understanding the interplay among scalar curvature, geometry, and topology. For a comprehensive overview of the subject, we refer the reader to Gromov's lecture notes \cite{Gromovfourlecture}.

In this work, we are particularly interested in the noncompact setting. For asymptotically flat manifolds, the positive mass theorem provides a fundamental link between scalar curvature and the geometry at infinity.

\begin{defn}\label{defn:AF}
A Riemannian manifold $(M^n,g)$ with $n\geq 3$ is said to be asymptotically flat \footnote{There are several inequivalent notions of asymptotic flatness in the literature, we use the classical one which is sufficient for our purpose.} if there exists a compact set $\Omega\Subset M$ such that $M\setminus \Omega$ is diffeomorphic to $\mathbb{R}^n \setminus \overline{B_{\euc}(0^n,1)}$, and in this coordinate chart, the metric $g_{ij}$ satisfies 
\begin{equation*}
\left\{
\begin{array}{ll}
\partial^k\left( g_{ij}-\delta_{ij}\right)=O(|x|^{-\sigma-k}),\;\;\text{for}\;\; k=0,1,2;\\[1mm]
 \mathrm{scal}(g)=O(|x|^{-q}),
\end{array}
\right.
\end{equation*}
as $x\to \infty$, for some $q>n$ and $\sigma>(n-2)/2$. In this case, the ADM mass of such manifold is defined to be 
$$m_{\mathrm{ADM}}(M,g):=\frac1{2(n-1)\omega_{n-1}}\lim_{r\to+\infty}\int_{\partial B_{\euc}(0^n,r)} (\partial_i g_{ij}-\partial_j g_{ii})\,\nu^j \,dA_{\euc},$$
where $\nu$ and $dA_{\euc}$ are the Euclidean unit normal and volume form of the sphere $\partial B_{\euc}(0^n,r)$, and $\omega_{n-1}$ is the volume of the unit sphere $\mathbb{S}^{n-1}$.
\end{defn}

The notion of mass appearing in Definition~\ref{defn:AF} was introduced by Arnowitt--Deser--Misner \cite{ArnowittDeserMisner} and is known as the ADM mass of $(M,g)$. Bartnik subsequently showed \cite{Bartnik1986} that, under the asymptotic flatness assumptions, the ADM mass is a geometric invariant of the Riemannian manifold, independent of the choice of asymptotically flat coordinates. The positive mass theorem states that for an asymptotically flat manifold with nonnegative scalar curvature, the ADM mass is nonnegative, and vanishes only when $(M^n,g)$ is isometric to Euclidean space $\mathbb{R}^n$. The theorem was first established by Schoen--Yau \cites{SchoenYau1979,Schoen1989} for dimensions $n\le 7$, and independently by Witten \cite{Witten1981} in the spin setting. In the non-spin case, the dimensional restriction arose from the difficulty of treating singularities in minimal hypersurfaces. This obstacle was recently overcome by Brendle--Wang \cite{BrendleWang2026}, building on the breakthrough work of Bi--Hao--He--Shi--Zhu \cite{BiHaoHeShiZhu2026}.

The rigidity statement of the positive mass theorem suggests that the vanishing of the mass reflects a strong propagation phenomenon: if the metric decays sufficiently rapidly at infinity, then the condition $\mathrm{scal}\geq 0$ forces the geometry to be flat not only near infinity but throughout the manifold. From this perspective, the mass serves as an obstruction to flatness at spatial infinity under the sole assumption $\mathrm{scal}\geq 0$. Motivated by this viewpoint, much effort has been devoted to extending the notion of mass to metrics with weaker regularity assumptions, either in their local regularity, their asymptotic decay, or both. A pioneering contribution in this direction was made by Lee--LeFloch \cite{LeeLeFloch2015}, who defined a generalized ADM mass for metrics of regularity $C^0 \cap W^{1,n}_{-q}$ and proved the corresponding positive mass theorem for spin manifolds.

For metrics with only $C^0$ asymptotic decay, a number of alternative notions of mass have been developed, especially in dimension three. One notable example is the isoperimetric mass introduced by Huisken \cite{Huisken}, which is defined through the asymptotic behavior of the isoperimetric profile. The associated positive mass theorem was later observed by Jauregui--Lee--Unger \cite{JaureguiLeeUnger}, and a related rigidity theorem was established by Benatti--Fogagnolo--Mazzieri \cite{BenattiFogagnoloMazzieri2025}. We also refer readers to \cite{Jauregui} and \cite{BenattiFogagnoloMazzieri20252023} for some  variation of Huisken's isoperimetric mass. More recently, Mazurowski--Yao \cite{MazurowskiYaoContinuous2026} proposed a notion of harmonic mass for metrics exhibiting very weak $C^0$ decay at infinity and developed the corresponding positive mass Theorem. Their harmonic mass can also be realized as the limit of the $C^0$ local mass, another notion of mass, introduced by Burkhardt-Guim \cite{Burkhardt2024}.

The primary goal of this work is to investigate the rigidity of Euclidean space using only the asymptotic $C^0$ behavior of the metric, rather than any notion of mass. For the Schwarzschild metric, the mass is reflected in the leading-order asymptotic term, which decays like $|x|^{2-n}$ as $|x|\to\infty$. In comparison with the decay assumptions made in Definition~\ref{defn:AFF}, it naturally raised the question of whether the decay rate $|x|^{2-n}$ is the natural physical threshold for mass rigidity. On the other hand, it is widely believed that the condition $\mathrm{scal}\geq 0$ depends only on the $C^0$ geometry of the metric, a principle manifested in Gromov's scalar-curvature stability theorem under $C^0$ convergence \cite{Gromov2014}; see also \cites{Bamler2016,Burkhardt2019,FogagnoloGattiPluda,MazurowskiYao2026,Lee2026}. Motivated by all these, Gromov asked whether the rigidity conclusion of the positive mass theorem remains valid for general metrics on $\mathbb{R}^n$ whose  $C^0$-deviation from the Euclidean metric decays faster than the Schwarzschild rate \cite{Gromovfourlecture}*{Section~3.11}. More precisely, he posed the following Euclidean $C^0$-Rigidity conjecture:
\begin{conj}[Gromov]\label{conj:gromov-endpoint}
Suppose $g$ is a smooth complete metric on $\mathbb R^n$ with $n\geq3$ such that $\mathrm{scal}(g)\ge0$ and
\[
        |g-g_{\mathrm{euc}}|(x)=o(|x|^{2-n}),
\]
as $x\to\infty$,  then $(\mathbb R^n,g)$ is isometric to Euclidean space.
\end{conj}
\begin{rem}
The decay rate stated in \cite{Gromovfourlecture}*{Section~3.11} contains a typographical error and is correct only in dimension $n=3$. For $n\geq 4$, the Schwarzschild decay rate is $|x|^{2-n}$ rather than $|x|^{-1}$. See Appendix~\ref{sec:append} for examples of non-flat metrics on $\mathbb{R}^n$ with decay rate $O(|x|^{2-n})$, which in particular satisfy $O(|x|^{2-n})\leq o(|x|^{-1})$ when $n\geq 4$. This observation follows the discussion in \cite{MazurowskiYaoRigidity2026}*{Remark 4}. 
\end{rem}

To the best of our knowledge, the first rigidity result in this direction was obtained by Appleton \cite{Appleton}, who considered metrics that are $L^\infty$ perturbations of the Euclidean metric with $|g-g_{\euc}|\in L^p(\mathbb{R}^n)$, for $p\in [1,\frac{n}{n-2})$. More recently, there has been significant progress on Conjecture~\ref{conj:gromov-endpoint} in dimension three. For example, under the slightly stronger decay assumption $O(|x|^{-1-\sigma})$ for $\sigma>0$, the conjectured rigidity follows from the isoperimetric mass rigidity theorem of Benatti--Fogagnolo--Mazzieri \cite{BenattiFogagnoloMazzieri2025}. More recently, Mazurowski--Yao \cite{MazurowskiYaoRigidity2026}, building on harmonic function techniques and a monotonicity formula of Agostiniani--Mazzieri--Oronzio \cite{AgostinianiMazzieriOronzio}, proposed a new approach to the problem. Shortly thereafter, You--Zhang \cite{YouZhang2026} refined the method of \cite{MazurowskiYaoRigidity2026} and established Conjecture~\ref{conj:gromov-endpoint} in dimension three. Around the same time, Mazurowski--Yao \cite{MazurowskiYaoContinuous2026} independently announced the same rigidity result as a consequence of their $C^0$ positive mass theorem. The arguments in \cites{MazurowskiYaoContinuous2026,YouZhang2026} rely crucially on exploiting the Gauss--Bonnet theorem on level sets and are therefore inherently tied to low dimensions. The main result of this paper complements these works by settling the conjecture in the remaining dimensions $n\geq 4$:

\begin{thm}\label{thm:gromov-conj}
Suppose $(M^n,g_0)$ is a smooth manifold such that  $n\geq4$, $\mathrm{scal}(g_0)\geq 0$ and there exists  a compact set $\Omega\Subset M$, $\hat r>0$ and a diffeomorphism $\Phi:\mathbb{R}^n\setminus \overline{B_{\euc}(0^n,\hat r)}\to  M\setminus \Omega$ such that 
\begin{equation}
|\Phi^*g_0-g_{\euc}|=o(|x|^{2-n})
\end{equation}
as $x\to\infty$, then $(M^n,g_0)$ is isometric to the Euclidean space.
\end{thm}
In particular restricting to $M=\mathbb{R}^n$, we answer Conjecture~
\ref{conj:gromov-endpoint} affirmatively for $n\geq 4$. Indeed our methodology also applies when $M$ a-priori has multiple asymptotically flat end, and forces $M$ to be one-ended.

Our approach combines Ricci flow smoothing with conformal deformation methods\footnote{In the first version of this work, we utilize spinor method instead.}. It is partly inspired by the positive mass theorem of Lee--LeFloch \cite{LeeLeFloch2015}, where the generalized ADM mass can be computed by weak asymptotic of the metric. The main technical challenge in our setting is the complete absence of any $C^1$ asymptotic control. We show, using the stability theory of the Ricci flow and parabolic regularity, that the initial metric $g_0$ can nevertheless be regularized to an asymptotically flat metric in a stronger sense. A crucial ingredient is that the decay rate $|x|^{2-n}$ is strictly faster than the standard asymptotic-flatness threshold $|x|^{\frac{2-n}{2}}$ appearing in Definition~\ref{defn:AF}. This aspect of the argument will be discussed in Section~\ref{sec:RDF}. Once this additional regularity is established, we are able to adapt the rigidity argument from the positive mass theorem to our setting. In fact in the case of spin and $n\geq 5$, Theorem~\ref{thm:gromov-conj} follows directly from the rigidity in the positive mass theorem of Lee--LeFloch \cite{LeeLeFloch2015} applied to the regularized metric; and for $n\geq 7$, it follows from the rigidity in positive mass theorem.  In Section~\ref{sec:dedecay}, we develop a unified approach that treats all dimensions $n\geq 4$. It is noteworthy that our method is critical with respect to dimension: it applies precisely for $n\geq 4$ and breaks down in dimension three.

The heart of our solution to Theorem~\ref{thm:gromov-conj} is based on smoothing with respect to an asymptotically flat metric, so that conformal deformation methods become available. The argument adapts readily to continuous metrics with the same decay rate and $\mathrm{scal}\geq 0$ in the sense of approximation. 

\medskip

{\it Acknowledgment:} The authors would like to thank Xuan Yao for some fruitful discussion. J. Chu was partially supported by National Key R\&D Program of China 2024YFA1014800 and 2023YFA1009900, and NSFC grants 12471052 and 12271008.  M.-C. Lee is supported by Hong Kong RGC grants No. 14300623 and No. 14304225, and an Asian Young Scientist Fellowship. J. Wan is supported by ERC-2023 AdG 101141855 BLaHST.

\section{Ricci-DeTurck flow and stability estimates}

In this section, we will collect some preliminary results on Ricci flow smoothing.   We will use a gauged version, i.e. the Ricci-DeTurck flow, with respect to some background Ricci flow $\tilde g(t)$. Given a smooth manifold $M^n$, the Ricci flow $\tilde g(t)$ on $M\times [0,T]$ is an one-parameter family of metric satisfying 
\begin{equation}
\partial_t \tilde g_{ij}=-2\tilde R_{ij}
\end{equation}
on $M\times [0,T]$, which is a weakly parabolic system of PDE. The Ricci-DeTurck flow with respect to $\tilde g(t)$ is a one parameter family of metric $g(t)$ solving 
\begin{equation}\label{eqn:RDF1}
\left\{
\begin{array}{ll}
\partial_t g_{ij}=-2R_{ij}+\nabla_i W_j+\nabla_j W_i,\\[1mm]
W^k=g^{ij}\left(\Gamma_{ij}^k-\tilde\Gamma_{ij}^k \right)
\end{array}
\right.
\end{equation}
on $M\times [0,T]$. Unlike the Ricci flow, the Ricci-DeTurck flow is a strictly parabolic system and is diffeomorphic to a Ricci flow $G(t):=\Psi_t^* g(t)$ with $G(0)=g(0)$, if $\Psi_t$ is an one parameter family of diffeomorphism such that
\begin{equation}\label{eqn:RD-ODE}
\left\{
\begin{array}{ll}
\partial_t \Psi_t(x)=-W\left( \Psi_t(x),t\right),\\[1mm]
\Psi_0(x)=x.
\end{array}
\right.
\end{equation}

We need the following weak stability of Ricci flow, proved by Chan, Lai and the second named author \cite{ChanLaiLee2025}, see also \cites{Burkhardt2019,CaiWang2026,ChuLee2025,KochLamm2012,Simon} for some earlier and related works. We use this version so that some constants to be chosen are more transparent.

\begin{thm}\label{thm:stability}
For any $\a>0$ and $n\geq 3$, there exists $\bar \e_0(n,\a),\bar L_0(n,\a)>0$ such that the following is true.  Suppose  $(M,\tilde g(t))$ is a smooth complete Ricci flow  on $M\times [0,T]$ such that for some $\bar r>0$, $\tilde g(t)$ satisfies 
\begin{enumerate}\setlength{\itemsep}{1mm}
\item[(a)] $|\Rm(\tilde g(t))|\leq \a t^{-1}$;
\item[(b)] $\mathrm{inj}(\tilde g(t))\geq \sqrt{\a^{-1}t}$;
\item[(c)] $\mathrm{Rm}(\tilde g(t))\geq -\bar r^{-2}$,
\end{enumerate}
on $M\times (0,T]$. If $g_0$ is a smooth complete metric on $M$ such that 
\begin{equation}\label{est0}
\|g_{0}-\tilde g(0)\|_{L^\infty(M,\tilde g(0))}\leq \e\leq \bar \e_0,
\end{equation}
then \eqref{eqn:RDF1} admits a smooth solution on $M\times [0,T\wedge \bar  r^2]$ such that 
\begin{equation}\label{estt}
    \|g(t)-\tilde g(t)\|_{L^\infty(M,\tilde g(t))}\leq \bar L_0 \e.
\end{equation}
Furthermore, if $g(t)$ and $\hat g(t)$ are both solutions to \eqref{eqn:RDF1} such that \eqref{estt} holds, then 
\begin{equation}
    \sup_{M\times [0,T\wedge \bar r^2]}\|g(t)-\hat g(t)\|_{L^\infty(M,\tilde g(t))}\leq \bar L_0 \cdot  \|g(0)-\hat g(0)\|_{L^\infty(M,\tilde g(0))}.
\end{equation}
\end{thm}
\begin{proof}
This follows from \cite{ChanLaiLee2025}*{Theorem 1.3} and parabolic scaling.
\end{proof}

Throughout this work, we will fix $\a=1$ in Theorem~\ref{thm:stability} and denote the corresponding constants:
\begin{equation}\label{eqn:constants-from-stab}
\e_0(n):=\bar\e_0(n,1),\quad\; L_0(n):=\bar L_0(n,1).
\end{equation}
Here we might assume $L_0 \e_0$ to be small by shrinking $\e_0$ further.

\section{Choice of the Background Ricci flow}

We want to regularize the metric given in Theorem~\ref{thm:gromov-conj} using the Ricci-DeTurck flow using Theorem~\ref{thm:stability}. To this end, we need to fix the choice of $\tilde g(t)$ in Theorem~\ref{thm:stability}. When $M=\mathbb{R}^n$, it is tempting to choose $\tilde{g}(t)=g_{\euc}$. However, the dacay assumption $\|g_{0}-g_{\euc}\|(x)=o(|x|^{2-n})$ does not guarantee \eqref{est0} in general. Instead, we construct a new metric $\tilde{g}_{0}$ and consider the Riccci flow starting from it. Throughout this section, we will assume $(M,g_0)$ is asymptotically flat in the following sense:

\begin{defn}\label{defn:AFF}
We say that $(M^n,g)$ is asymptotically flat if  there exists a compact set $\Omega\Subset M$ and $\hat r>0$ such that $M\setminus \Omega$ is diffeomorphic to $\mathbb{R}^n\setminus \overline{B_{\euc}(0^n,\hat r)}$ and, in the chart it satisfies 
\begin{equation}
\|g_{0}-g_{\euc}\|_{L^\infty(\partial B_{\euc}(0^n,r))}:=\e(r)
\end{equation}
where $\e(r)=o(1)$ as $r\to+\infty$. For notation convenience, we abbreviate 
$$M= \Omega \cup \left(\mathbb{R}^n\setminus \overline{B_{\euc}(0^n,\hat r)}\right).$$
\end{defn}

For a asymptotically flat manifold $(M^n,g_0)$ in the sense of Definition~\ref{defn:AFF}, we truncate it so that it is exactly a flat metric outside compact set. For the purpose of applying stability, we need a quantitative version of truncation.

\begin{lma}\label{lma:ref-initial}
Suppose $(M^n,g_0)$ is a asymptotically flat in the sense of Definition~\ref{defn:AFF},  then there exists a metric $\tilde g_0$ on $M$ and $r_0>\hat r,\b_0>0$ such that 
\begin{enumerate}\setlength{\itemsep}{1mm}
\item[(i)] $\tilde g_0=g_{\euc}$ outside $B_{\euc}(0^n,2r_0)$;
\item[(ii)] $\tilde g_0=g_0$ on $ \Omega\cup A_{\euc}(0^n,\hat r,r_0)$  where $A_{\euc}$ denotes the Euclidean annulus;
\item[(iii)] $(1-\e_0) g_{\euc}\leq \tilde g_0\leq (1+\e_0)  g_{\euc}$ on $A_{\euc}(0^n,r_0,2r_0)$;
\item[(iv)]  $(1-\e_0) g_{\euc}\leq  g_0\leq (1+\e_0) g_{\euc}$ outside $B_{\euc}(0^n,r_0)$;
\item[(v)] $\sup_{M}|\Rm(\tilde g_0)|\leq \b_0^2$ and $\inf_{M}\mathrm{inj}(\tilde g_0)\geq \b_0^{-1}$,
\end{enumerate}
where $\e_0$ is the constant from \eqref{eqn:constants-from-stab}.
\end{lma}
\begin{proof}

By assumption, there exists $r_0$ large enough such that outside $B_{\euc}(0^n,r_0)$ we have 
\begin{equation}
(1-\e_0) g_{\euc}\leq g_0\leq (1+\e_0)g_{\euc}.
\end{equation}
This fixes the choice of $r_0$ for (iv).

We let $\phi:M^n\to [0,1]$ be a smooth function such that $\phi=1$ on $A_{\euc}(0^n,\hat r, r_0)\cup \Omega$ and vanishes outside $B_{\euc}(0^n,2r_0)$.  We consider the metric 
\begin{equation}\label{eqn:defn-ref-euc}
\tilde g_0:= \phi \cdot g_0 +(1-\phi) g_{\euc}
\end{equation}
on $M^n$. Clearly, conclusions (i)-(iii) holds by our choice of $\phi$ and $r_0$. Conclusion (v) follows from compactness, since $\tilde g_0$ is flat outside a compact set.
\end{proof}

The background Ricci flow will be chosen to be the (unique) Ricci flow starting from $\tilde g_0$, obtained from Lemma~\ref{lma:ref-initial}.
\begin{lma}\label{lma:choice-ref-flow}
Suppose $\tilde g_0$ is the metric obtained from Lemma~\ref{lma:ref-initial}, then there exists a short-time solution $\tilde g(t)$ to the Ricci flow on $M^n\times [0,T_n\b_0^{-2}]$ for some dimensional constant $T_n>0$, such that $\tilde g(0)=\tilde g_0$,
\begin{enumerate}\setlength{\itemsep}{1mm}
\item $|\Rm(\tilde g(t))|\leq 2\b_0^2$;
\item $\frac12 \tilde g_0\leq \tilde g(t)\leq 2 \tilde g_0$;
\item $\mathrm{inj}(\tilde g(t))\geq c_n\b_0^{-1}$
\end{enumerate}
on $M^n\times [0,T_n\b_0^{-2}]$. In particular, $\tilde g(t)$ satisfies the conditions (a)-(c) in Theorem~\ref{thm:stability} on $(0,T]$, for some $T>0$ and $\bar  r>0$.
\end{lma}
\begin{proof}
Let $\tilde g_0$ be the metric obtained from Lemma~\ref{lma:ref-initial}. Thanks to (v), it has bounded curvature and hence there exists a short-time solution $\tilde g(t)$ on $M^n\times [0,T_{\max})$, by Shi \cite{Shi1989}. Furthermore, it has bounded curvature and $T_{\max}$ is characterized by the maximal time such that the curvature remains bounded globally.

It is well-known (for example see \cite{ChowBook2}*{section 12}) that the curvature $\Rm$ under the Ricci flow evolves by 
\begin{equation}\label{eqn:evo-Rm}
\left(\partial_t-\Delta_{\tilde g(t)} \right)|\Rm(\tilde g)|^2\leq C_n|\Rm(\tilde g)|^3.
\end{equation}
Since Shi's solution has bounded curvature, the curvature estimates for $t\in [0,T_{max}\wedge c_n \b_0^{-2}]$ for a small dimensional constant $c_n>0$, follows from maximum principle. This also implies $T_{\max}>c_n \b_0^{-2}$. The metric equivalence and injectivity radius estimate follow from the curvature bound and integrating the metric using Ricci flow equation.
\end{proof}

Without loss of generality, we will assume $T\leq \bar r^2$ in Lemma~\ref{lma:choice-ref-flow}.  We will be choosing $\tilde g(t)$ from Lemma~\ref{lma:choice-ref-flow} as our background Ricci flow, appeared in Theorem~\ref{thm:stability}. Due to the parabolic nature, $\tilde g(t)$ is not necessarily flat outside compact set, even though $\tilde g_0=g_{\euc}$ outside $B_{\euc}(0^n,2r_0)$. 

The next lemma shows that at spatial infinity, $\tilde g(t)$ is still close to a flat metric in a strong sense. This is in similar spirit to \cite{Chen2009}*{Corollary 3.2}
\begin{lma}\label{lma:almostflat}
Suppose $\tilde g(t)$ is a smooth solution to the Ricci flow on $M\times [0,T]$ and $x_0\in M,r>0$ such that 
\begin{enumerate}\setlength{\itemsep}{1mm}
\item $B_{\tilde g_0}(x_0,r)\Subset M$;
\item $\mathrm{Rm}(\tilde g(0))=0$ on $B_{\tilde g_0}(x_0,r)$;
\item $|\Rm(\tilde g(t))|\leq \a t^{-1}$ on $B_{\tilde g_0}(x_0,r)\times (0,T]$,
\end{enumerate}
then for all $\ell\in \mathbb{N}$, there exists $T_\ell(n,\a)>0$ such that for all $t\in [0,T\wedge T_\ell r^2]$, 
\begin{equation*}
\left\{
\begin{array}{ll}
|\Rm(\tilde g(x_0,t))|\leq t^\ell r^{-2\ell-2};\\[1mm]
|\tilde \nabla \Rm(\tilde g(x_0,t))|\leq t^\ell r^{-2\ell-3}.
\end{array}
\right.
\end{equation*} 
\end{lma}
\begin{proof}
By rescaling, we might assume $r=1$. The curvature estimate follows from \eqref{eqn:evo-Rm} and \cite{LeeTam2022}*{Theorem 1.1}. The higher order estimate is completely identical. By Shi's type estimate \cite{Shi1989} and the improved estimate of curvature $|\Rm(\tilde g(t))|$, we have $|\tilde\nabla\widetilde\Rm|\leq C_nt^{-1/2}$ locally on a fixed time interval. By noting that 
\begin{equation}
\begin{split}
\left(\partial_t-\Delta_{\tilde g(t)} \right) |\tilde\nabla \widetilde\Rm|^2
\leq C_n|\widetilde\Rm||\tilde\nabla \widetilde\Rm|^2
\end{split}
\end{equation}
from differentiating the curvature tensor along the Ricci flow, result follows from the same argument after we shrink $T_\ell$ slightly.
\end{proof}

\begin{rem}
We remark here that since $\tilde g_0$ is chosen to be $g_0$ in compact set and is a-priori  a non-flat metric, only the flatness outside compact set suffices. Indeed, Lemma~\ref{lma:almostflat} asserts that even though $\tilde g_0$ is non-flat globally, $\tilde g(t)$ will remain almost flat quantitatively on region where its initial metric $\tilde g_0$ is flat. 
\end{rem}

\section{Ricci-DeTurck flow from given metric}\label{sec:RDF}

Let $g_0$ be the metric given in the set-up in Theorem~\ref{thm:gromov-conj} and $\tilde g(t)$ be the Ricci flow from Lemma~\ref{lma:choice-ref-flow}. in order to make full use of the decay assumption $\|g_{0}-g_{\euc}\|(x)=o(|x|^{2-n})$, We will use the Ricci-DeTurck flow (the gauged version) with respect to $\tilde g(t)$ to regularize $g_0$. Throughout this work, the Ricci-DeTurck flow will be referring to that with respect to $\tilde g(t)$. We start with its existence.

\begin{prop}\label{prop:flowing-g_0}
Suppose $(M^n,g_0)$ is asymptotically flat in the sense of Definition~\ref{defn:AFF} satisfying  $\mathrm{scal}(g_0)\geq 0$. Let $\tilde g(t),t\in [0,T]$ be the Ricci flow obtained from Lemma~\ref{lma:choice-ref-flow}, then there exists a smooth solution to the  Ricci-DeTurck flow $g(t)$ with respect to $\tilde g(t)$ on $M^n\times [0,T]$, such that $ g(0)= g_0$ and $\mathrm{scal}(g(t))\geq 0$ on $M\times [0,T]$.
\end{prop}
\begin{proof}
By Lemma~\ref{lma:ref-initial}, the metric $g_0$ satisfies 
\begin{equation}
(1-\e_0) \tilde g_0\leq g_0\leq (1+\e_0) \tilde g_0
\end{equation}
on $M$. Thanks to the estimates in Lemma~\ref{lma:choice-ref-flow}, we apply Theorem~\ref{thm:stability} to $g_0$ and then obtain a smooth solution $g(t)$ to Ricci-DeTurck flow  starting from $g_0$ such that 
\begin{equation}
(1-L_0\e_0) \tilde g(t)\leq g(t)\leq (1+L_0\e_0) \tilde g(t)
\end{equation}
on $M\times [0,T]$, where $L_0$ is the constant from \eqref{eqn:constants-from-stab}. 

The non-negativity of $\mathrm{scal}(g(t))$ follows from scaling and localized maximum principle \cite{LeeTam2022}*{Theorem 1.1} by interchanging the Ricci-DeTurck flow with the Ricci flow via \eqref{eqn:RD-ODE}, see the proof of \cite{Lee2026}*{Lemma 3.1}.
\end{proof}

The next proposition shows that if $g_0$ decays to $g_{\euc}$ at some rate, its Ricci-DeTurck flow will preserve the decay rate.

\begin{prop}\label{prop:zero-th-decay-preserved}
Suppose $g_0$ in Proposition~\ref{prop:flowing-g_0} in addition satisfies 
$$\|g_{0}-g_{\euc}\|_{L^\infty(\partial B_{\euc}(0^n,r))}=o(r^{2-n})$$
as $r\to+\infty$, then the Ricci-DeTurck flow $g(t)$ from Proposition~\ref{prop:flowing-g_0} also satisfies 
$$\|g(t)-g_{\euc}\|_{L^\infty(\partial B_{\euc}(0^n,r))}=o(r^{2-n})$$
for all $t\in [0,T]$ uniformly.
\end{prop}
\begin{proof}
Let $r_0>0$ be the constant from Lemma~\ref{lma:ref-initial}. Let $\e>0$ and $R_\e>2r_0$ such that for all $s>R_\e$, 
\begin{equation}\label{eqn:decay-ep}
\|g_0-g_{\euc}\|_{L^\infty(\partial B_{\euc}(0^n,s))}<\e s^{2-n}.
\end{equation}

Fix an arbitrary $r>R_\e$. We first give an estimate with respect to the reference Ricci flow $\tilde g(t)$.

\begin{claim}\label{claim:esti-to-ref}
For all $\ell\in \mathbb{N}$, there exists $T_\ell(n)>0$ such that for all $x\in \mathbb{R}^n\setminus B_{\euc}(0^n,3r)$ and $t\in [0,T\wedge T_\ell r^2]$, we have 
\begin{equation}
\left\{
\begin{array}{ll}
|g(t)-\tilde g(t)|_{\tilde g(t)}\leq L_0\e r^{2-n}+t^\ell r^{-2\ell};\\[1mm]
|\tilde\nabla g(t)|_{\tilde g(t)}\leq L_1\e r^{2-n}t^{-1/2}+t^\ell r^{-2\ell-1},
\end{array}
\right.
\end{equation}
for some $L_1(n)>0$.
\end{claim}
\begin{proof}[Proof of Claim~\ref{claim:esti-to-ref}]
We let $\phi_r:M^n\to [0,1]$ be a smooth function such that $\phi_r=1$ on $A_{\euc}(0^n,\hat r,r)\cup\Omega$ and vanishes outside $B_{\euc}(0^n,2r)$. We define the following smooth metric $g_{0,r}$ on $M$: 
\begin{equation}
g_{0,r}:=\phi_r g_0 +(1-\phi_r)g_{\euc}=\phi_r g_0 +(1-\phi_r)\tilde g_{0},
\end{equation}
where in the second equality we have used $\tilde g_0=g_{\euc}$ outside $B_{\euc}(0^n,2r_0)$ and $r>2r_0$. By \eqref{eqn:decay-ep},  
\begin{equation}\label{eqn:compare-init-1}
\begin{split}
\|g_{0,r}-g_0\|_{L^\infty(M,\tilde g_0)}
&=\| (1-\phi_r)\cdot (g_0-g_{\euc})\|_{L^\infty(M,\tilde g_0)}\\
&\leq \|g_0-g_{\euc}\|_{L^\infty(\mathbb{R}^n\setminus B_{\euc}(0^n,r),\,\tilde g_0)}\\
&\leq \e\, r^{2-n}=: \e_r.
\end{split}
\end{equation}
Similarly since $\tilde g_0=g_0$ on $A_{\euc}(0^n,\hat r,r_0)\cup \Omega$, we have
\begin{equation}\label{eqn:compare-init-2}
\begin{split}
\|g_{0,r}-\tilde g_0\|_{L^\infty(M,\tilde g_0)}
&=\| \phi_r\cdot (\tilde g_0-g_0)\|_{L^\infty(M,\tilde g_0)}\\
&\leq \| \tilde g_0-g_0\|_{L^\infty(A_{\euc}(0^n,\hat r,2r)\cup\Omega,\tilde g_0)}\\
&= \| \tilde g_0-g_0\|_{L^\infty(A_{\euc}(0^n,r_0,2r),\tilde g_0)}< \e_0,
\end{split}
\end{equation}
where we have used (iv) in Lemma~\ref{lma:ref-initial} and \eqref{eqn:defn-ref-euc}.

In particular, \eqref{eqn:compare-init-2} allows us to construct a Ricci-DeTurck flow $g_r(t)$ on $M\times [0,T]$ starting from $g_{0,r}$, by Theorem~\ref{thm:stability}, while \eqref{eqn:compare-init-1} with Theorem~\ref{thm:stability} implies 
\begin{equation}\label{eqn:zero-th}
\|g_r(t)-g(t)\|_{L^\infty(M,\tilde g(t))}\leq L_0 \cdot \e_r= L_0 \e  r^{2-n},
\end{equation}
for all $t\in [0,T]$.

Let $x_0 \in M$ such that $B_{\euc}(x_0,r)\Subset \mathbb{R}^n\setminus B_{\euc}(0^n,2r)$. Since $r$ is large and $\tilde g_0=g_{\euc}$ outside $B_{\euc}(0^n,2r_0)$, $B_{\euc}(x_0,r)=B_{\tilde g_0}(x_0,r)$. Since $g_r(t)$ is a solution to the Ricci-DeTurck flow such that $g_r(0)=g_{\euc}=\tilde g_0$ on $B_{\tilde g_0}(x_0,r)$, by shrinking $\e_0$ (if necessary) we might apply \cite{Lee2025}*{Corollary 2.1} that for all $\ell\in \mathbb{N}$, there is $T_\ell(n)>0$ such that for all $t\in [0,T_\ell r^2]$,
\begin{equation}
(1-t^{\ell}r^{-2\ell}) \tilde g(x_0,t)\leq g_r(x_0,t) \leq (1+t^\ell r^{-2\ell}) \tilde g(x_0,t).
\end{equation}
In particular, this implies that for all $x\in \mathbb{R}^n\setminus  B_{\euc}(0^n,3r)$ and $t\in [0,T\wedge T_\ell r^2]$,
\begin{equation}
\begin{split}
|g(t)-\tilde g(t)|_{\tilde g(t)}
&\leq |g(t)-g_r(t)|_{\tilde g(t)}+|g_r(t)-\tilde g(t)|_{\tilde g(t)}\\
&\leq L_0\e r^{2-n}+t^\ell r^{-2\ell}.
\end{split}
\end{equation}
This proves the zero-th order estimate. The first order is similar using \cite{ChanLaiLee2025}*{Lemma 4.7} to obtain first order version of \eqref{eqn:zero-th}. 
\end{proof}

Next we compare the reference Ricci flow with the Euclidean metric. 
\begin{claim}\label{claim:esti-to-ref-euc}
For all $\ell\in \mathbb{N}$, there exists $T_\ell(n)>0$ such that for all $x\in\mathbb{R}^n\setminus B_{\euc}(0^n,3r)$ and $t\in [0,T\wedge T_\ell r^2]$, we have 
$$|\tilde g(t)-g_{\euc}|_{g_{\euc}} \leq t^\ell r^{-2\ell}.$$
\end{claim}
\begin{proof}[Proof of Claim~\ref{claim:esti-to-ref-euc}]
Fix $x_0\in\mathbb{R}^n\setminus B_{\euc}(0^n,3r)$ so that $B_{\tilde g_0}(x_0,r)\Subset \mathbb{R}^n\setminus B_{\euc}(0^n,2r)$, where $\tilde g_0=g_{\euc}$ on $B_{\tilde g_0}(x_0,r)$. The claim follows from Lemma~\ref{lma:almostflat} and integrating using the Ricci flow equation.
\end{proof}

We combine Claim~\ref{claim:esti-to-ref} and Claim~\ref{claim:esti-to-ref-euc} to conclude that as $r\to+\infty$, 
\begin{equation}
|g(t)-g_{\euc}|_{g_{\euc}}
\leq L_0\e r^{2-n}+2t^\ell r^{-2\ell}
\end{equation}
for any $x\in \mathbb{R}^n\setminus B_{\euc}(0^n,3r)$ and $t\in [0,T\wedge T_\ell r^2]$, where $\ell$ is arbitrary. Choosing $\ell$ sufficiently large, we complete the proof.
\end{proof}

A straight-forward modification of the proof of Proposition~\ref{prop:zero-th-decay-preserved} yields the first order estimate. This in turn is based on parabolic regularizing properties of Ricci-DeTurck flow.
\begin{prop}\label{prop:1st-decay-obtained}
Suppose $g_0$ satisfies the assumption in Proposition~\ref{prop:zero-th-decay-preserved}, then the Ricci-DeTurck flow $g(t)$ from Proposition~\ref{prop:flowing-g_0} satisfies 
$$t^{1/2}\cdot \|\nabla^{\euc}g(t)\|_{L^\infty(\partial B_{\euc}(0^n,r))}=o(r^{2-n})$$
for all $t\in (0,T]$ uniformly.
\end{prop}
\begin{proof}
It suffices to note that using the first order estimate of curvature from Lemma~\ref{lma:almostflat}, $\tilde g_0=g_{\euc}$ outside compact set, and the evolution equation of connection, i.e. $\partial_t \Gamma(\tilde g(t))=\tilde g^{-1}*\tilde\nabla\widetilde\Ric$, we have 
$|\Gamma(\tilde g(t))|\leq C_n t^\ell r^{-2\ell-1}$ for all $x\in\mathbb{R}^n\setminus B_{\euc}(0^n,3r)$ and $t\in [0,T\wedge T_\ell r^2]$, where $r\to+\infty$. Result follows by combining this with Claim~\ref{claim:esti-to-ref}. 
\end{proof}

\begin{rem}
By applying parabolic bootstrapping repeatedly, the $C^k$ regularity analogous to $C^1$ regularity appeared in Proposition~\ref{prop:1st-decay-obtained} holds for all $k\in \mathbb{N}$ and thus, the regularized metric $g(t)$ will be $C^k$-asymptotically flat for  $0\leq k<\frac{n-2}2$. The estimate up to $k=1$ will be sufficient for our main purpose. From this, we also see that $n=3$, where $k$ can only be $0$, will be a threshold in applying rigidity in positive mass theorem type result.
\end{rem}

\section{Rigidity of metrics with fast decay}\label{sec:dedecay}

In this section, we finish up the proof of Theorem~\ref{thm:gromov-conj}. We first prove the following weaker version which is motivated by the rigidity in positive mass theorem with weak (weighted) regularity at spatial infinity.

\begin{thm}\label{thm:rigidity-ngeq4}
Suppose $(M^n,g)$ is a asymptotically flat in the sense of Definition~\ref{defn:AFF} such that $n\geq 4$ and 
\begin{enumerate}\setlength{\itemsep}{1mm}
\item[(i)] $\mathrm{scal}(g)\geq 0$;
\item[(ii)] $|g-g_{\euc}|(x)=o(|x|^{2-n})$ as $x\to \infty$;
\item[(iii)] $|\partial g|(x)=o(|x|^{2-n})$ as $x\to \infty$,
\end{enumerate}
then $\mathrm{scal}(g)\equiv 0$ on $M$.
\end{thm}

Intuitively speaking, Theorem~\ref{thm:rigidity-ngeq4} asserts that the mass, in a generalized sense,  vanishes so that positive mass Theorem suggests that the metric is flat. This will be proved using idea\footnote{In the first version of this work, we rely method in spin geometry building on analysis in \cites{Bartnik1986,LeeLeFloch2015,LeeGeometricRelativity,
ParkerTaubes1982,Witten1981}.} in the rigidity part of positive mass Theorem.

\begin{proof}[Proof of Theorem~\ref{thm:rigidity-ngeq4}]

We first approximate $g$ by metric which is flat outside compact set. For any $R>0$ large, we let $\phi_R:M^n\to [0,1]$ be a smooth function such that $\phi_R=1$ on $A_{\euc}(0^n,\hat r,R)\cup \Omega$, vanishes outside $B_{\euc}(0^n,2R)$ and 
\begin{equation}
|\partial \phi_R|^2+|\partial^2 \phi_R|\leq 10^4 R^{-2}.
\end{equation}
We consider the smooth metric 
\begin{equation}\label{eqn:metric-trunca}
g_R:= g_{\euc}+\phi_R\cdot (g-g_{\euc})
\end{equation}
on $\mathbb{R}^n\setminus B_{\euc}(0^n,\hat r)$ so that $g_R=g$ on $A_{\euc}(0^n,\hat r,R)$ and $g_R=g_{\euc}$ outside $B_{\euc}(0^n,2R)$. We also extend $g_R$ trivially by $g$ on $\Omega$.

By compactness and metric decay, we might find $\Lambda>1$ such that global $L^2$ Sobolev inequality holds for all $(M^n,g_R)$ as $R\to+\infty$: there exists $\Lambda>0$ such that for all $R\to+\infty$ and $\phi\in C^\infty_c(M)$, we have
\begin{equation}\label{eqn:Sobolev}
\left(\int_M \phi^\frac{2n}{n-2}\,d\mathrm{vol}_{g_R}\right)^\frac{n-2}{n}\leq \Lambda \cdot\left(\int_M |\nabla^{g_R}\phi|^2\,d\mathrm{vol}_{g_R} \right).
\end{equation}
It is also clear that $m_{\mathrm{ADM}}(g_R)=0$ since it is Euclidean  outside compact set.

\medskip

 For convenience, we also write 
\begin{equation}
\e_R:= R^{n-2} \left( \|h\|_{L^\infty(A_{\euc}(0^n,R,2R))}+\|\partial h\|_{L^\infty(A_{\euc}(0^n,R,2R))} \right)\in [0,1)
\end{equation}
where $h:=g-\delta$ and $g_{\euc}=\delta$. From this, it is also harmless to assume 
\begin{equation}\label{eqn:metricc}
\frac12 g_{\euc}\leq g\leq 2 g_{\euc}
\end{equation}
outside compact set.
\medskip

\begin{claim}\label{claim:existence-sol}
There exists $R_0>0$ such that for all $R>R_0$, we can find smooth positive function $\psi_R$ satisfying 
\begin{equation}\label{eqn:equationss}
\left\{
\begin{array}{ll}
-\Delta_{g_R}\psi_R+c_0(n) \cdot \mathrm{scal}(g_R)\cdot \psi_R=0;\\[2mm]
\lim_{x\to\infty} \psi_R=1,
\end{array}
\right.
\end{equation}
where $c_0(n):=\frac{n-2}{4(n-1)}$. Moreover, $\psi_R=1+A_R |x|^{2-n}+\omega_R$ for a constant $A_R$ and function $\omega_R$ with $|\omega_R|=O(|x|^{1-n})$ and $|\partial \omega_R|=O(|x|^{-n})$ as $x\to\infty$.
\end{claim}
\begin{proof}[Proof of Claim~\ref{claim:existence-sol}]
It suffices to solve 
\begin{equation}
\left\{
\begin{array}{ll}
\displaystyle-\Delta_{g_R}v_R+c_0 \cdot \mathrm{scal}(g_R)\cdot v_R=-c_0 \cdot \mathrm{scal}(g_R);\\[2mm]
\lim_{x\to\infty}v_R=0.
\end{array}
\right.
\end{equation}
By following \cite{JiangShengZhang2022}*{Lemma 3.1} (which is building on \cite{SchoenYau1979}*{Lemma 3.2}), it suffices to show that for $R$ sufficiently large and $\psi\in C^\infty_c(M)$,
\begin{equation}\label{eqn:energy}
\int_M \mathrm{scal}(g_R)\cdot \psi^2\,d\mathrm{vol}_{g_R}\geq -\e_0\int_M |\nabla^{g_R}\psi|^2\,d\mathrm{vol}_{g_R},
\end{equation}
where $\e_0$ depends only on the Sobolev constant $\Lambda$ of $(M,g_R)$ in \eqref{eqn:Sobolev}.

%
%
\medskip
%

We now claim that the scalar curvature appeared in \eqref{eqn:energy} is almost non-negative. It suffices to control the integral on the left hand side in \eqref{eqn:energy} over the region $A_{\euc}(0^n,R,2R)$, since $\mathrm{scal}(g_R)\geq 0$ outside. Motivated by \eqref{eqn:metric-trunca}, we aim to control
\[
 \mathrm{scal}(g_R) -  \phi_R\cdot \mathrm{scal}(g).
\]
It suffices to consider the region where $\phi_R\in (0,1]$. In particular, we might use the divergence form of scalar curvature with respect to Euclidean background metric in standard coordinate $g_{\euc}=\delta$:
\begin{equation}\label{eqn:scal-div}
\begin{split}
\mathrm{scal}(g)=\partial_k V^k+F,
\end{split}
\end{equation}
where 
\begin{equation*}
\left\{
\begin{array}{ll}
V^k=g^{ij}g^{kl} \left(\partial_j g_{il}-\partial_l g_{ij}\right),\\[3pt]
F=g^{-1}*g^{-1}*g^{-1} *\partial g *\partial g,
\end{array}
\right.
\end{equation*}
for example see \cite{LeeLeFloch2015}*{(2.2)-(2.4)}. For $V^k$ evaluated at $g_R:=\delta+h_R$ where $h_R=\phi_R h$, we have
\begin{equation}\label{eqn:linearzV}
\begin{split}
W_R^k&:= V^k(g_R)-\phi_R V^k(g)\\
&= g_R^{ij}g_R^{kl} \left[\partial_j (h_R)_{il}-\partial_l (h_R)_{ij}\right]-\phi_R g^{ij}g^{kl} \left(\partial_j h_{il}-\partial_l h_{ij}\right)\\
&=g_R^{-1}*g_R^{-1}* \partial \phi_R *h+\phi_R (1-\phi_R) * g_R^{-1}* g_R^{-1}*g^{-1}*h*\partial h\\
&\quad +\phi_R (1-\phi_R) * g_R^{-1}* g^{-1}*g^{-1}*h*\partial h.
\end{split}
\end{equation}
Using the decay assumptions, the above shows
\begin{equation}\label{eqn:W-error}
|V^k(g)|+R\cdot |W^k_R| \leq C_1\e_R R^{2-n}
\end{equation}
on $A_{\euc}(0^n,R,2R)$, for some uniform constant $C_1>0$. 

\medskip

By the definition of $W_{R}^{k}$ from \eqref{eqn:linearzV}, we see that $W_{R}^{k}$ vanishes outside $A_{\euc}(0^{n},R,2R)$. Similarly, 
\begin{equation}\label{eqn:F-cutt}
\begin{split}
& |F(g_R)-\phi_R F(g)|\\
\leq {} & C_2|\partial h|^2+C_2R^{-1} |h|\cdot|\partial h|+C_2 R^{-2} |h|^2 \\
\leq {} & C_3 \e_R R^{4-2n},
\end{split}
\end{equation}
and it also vanishes outside $A_{\euc}(0^{n},R,2R)$. Thus
\begin{equation}\label{eqn:F-error}
|F(g_R)-\phi_R F(g)| \leq C_3 \e_R R^{4-2n}\chi_{A_{\euc}(0^n,R,2R)}.
\end{equation}
Combining \eqref{eqn:W-error} and \eqref{eqn:F-error}, we conclude that 
\begin{equation}\label{eqn:scal-err 1}
\begin{split}
&  \mathrm{scal}(g_R)-\phi_R \cdot\mathrm{scal}(g)\\
= {} & \left[\partial_k V^k(g_R) - \phi_R\cdot  \partial_k V^k(g)\right] +\left(F(g_R) -\phi_R F(g)\right)\\
= {} & \left[\partial_k (\phi_RV^k(g)+W^k_R) - \phi_R\cdot  \partial_k V^k(g)\right] +\left(F(g_R) -\phi_R F(g)\right)\\
\geq {} & \partial_k W^k_R- |\partial_k \phi_R| |V^k(g)|-\left|F(g_R) -\phi_R F(g)\right|\\
\geq {} & \partial_k W^k_R-C_4\e_R R^{1-n}\chi_{A_{\euc}(0^n,R,2R)}.
\end{split}
\end{equation}
In particular, using  \eqref{eqn:scal-err 1} and $\mathrm{scal}(g)\geq 0$,
\begin{equation}\label{eqn:I-II-III}
\begin{split}
&\quad \int_M \mathrm{scal}(g_R) \cdot \psi^2\,d\mathrm{vol}_{g_R}\\
&\geq \int_{M} \left(\mathrm{scal}(g_R) -\phi_R \cdot\mathrm{scal}(g)\right)\cdot \psi^2\,d\mathrm{vol}_{g_R}\\
&= \int_{M} \psi^2\cdot  \partial_k W^k_R \,d\mathrm{vol}_{g_R}-C_5 \e_R R^{1-n}\int_{A_{\euc}(0^n,R,2R)} \psi^2 \,d\mathrm{vol}_{g_R} \\
&=:\mathbf{I}-\mathbf{II}.
\end{split}
\end{equation}
For $\mathbf{II}$, \eqref{eqn:Sobolev} and the metric equivalence imply
\begin{equation}\label{eqn:III-by-I}
\begin{split}
\mathbf{II}&\leq C_6 \e_R R^{3-n} \left(\int_{M} \psi^\frac{2n}{n-2}\,d\mathrm{vol}_{g_R} \right)^\frac{n-2}{n}\\
&\leq C_6\Lambda\e_R R^{3-n}\cdot \int_M |\nabla^{g_R}\psi|^2\,d\mathrm{vol}_{g_R}.
\end{split}
\end{equation}
For $\mathbf{I}$, we first compare the Euclidean-divergence with $g_R$-divergence: 
\begin{equation}
\mathrm{div}_{g_{\euc}}W_R=\mathrm{div}_{g_{R}}W_R-\Gamma(g_R)^k_{kl}\cdot  W^l_R
\end{equation}
so that Stokes' Theorem implies
\begin{equation}\label{eqn:II-by-I}
\begin{split}
|\mathbf{I}|&= \left|\int_{M} \partial_k W^k_R \cdot \psi^2 \,d\mathrm{vol}_{g_R}\right|\\
&\leq \left|\int_{A_{\euc}(0^n,R,2R)}  W_R^k \cdot \partial_k\psi^2 \,d\mathrm{vol}_{g_R}\right|+C_7 \e_R R^{3-2n} \int_{A_{\euc}(0^n,R,2R)} \psi^2\,d\mathrm{vol}_{g_R}\\[1mm]
&\leq C_8\e_R R^{3-2n} \|\psi\|^2_{L^2(A_{\euc}(0^n,R,2R),g_R)}\\[2mm]
&\quad +C_8\e_R R^{1-n} \|\psi\|_{L^2(A_{\euc}(0^n,R,2R),g_R)}\cdot \|\nabla^{g_R}\psi\|_{L^2((A_{\euc}(0^n,R,2R),g_R)}\\[1mm]
&\leq C_9 \Lambda \e_R R^{2-n} \cdot \int_M |\nabla^{g_R}\psi|^2\,d\mathrm{vol}_{g_R},
\end{split}
\end{equation}
where we have used argument similar to \eqref{eqn:III-by-I} to control the $L^2$ norm over annulus by its gradient.

By substituting \eqref{eqn:II-by-I} and \eqref{eqn:III-by-I} into \eqref{eqn:I-II-III}, we deduce that 
\begin{equation}\label{eqn:I-controlled}
\int_M \mathrm{scal}(g_R)\cdot \psi^2\,d\mathrm{vol}_{g_R}\geq -\left(C_{9}\Lambda\e_R R^{2-n}+C_6\Lambda \e_R R^{3-n}\right)\cdot\int_M |\nabla^{g_R}\psi|^2\,d\mathrm{vol}_{g_R}.
\end{equation}
In particular for $R$ sufficiently large, the coefficient can be made arbitrarily small, thanks to $n\geq 3$.   This verifies \eqref{eqn:energy} and thus shows the existence of $\psi_R$ by Fredholm alternative. The asymptotic behaviour estimate of $\psi_R$ can be proved by following that of \cite{SchoenYau1979}*{Lemma 3.2 \& Lemma 3.3}.
\end{proof}

\medskip

The next shows that 
the function $\psi_R$ obtained from Claim~\ref{claim:existence-sol}  satisfies an additional inequality.
\begin{claim}\label{claim:massss}
For $R$ sufficiently large, we have 
\begin{equation}\label{eqn:energgy}
\int_M |\nabla^{g_R}\psi_R|^2+c_0 \cdot \mathrm{scal}(g_R)\cdot \psi_R^2\,d\mathrm{vol}_{g_R}\leq 0.
\end{equation}
\end{claim}
\begin{proof}[Proof of Claim~\ref{claim:massss}]

We follow the idea in \cite{SchoenYau1979}*{Corollary 3.1}. We consider the metric $$\tilde g_R:=\psi_R^\frac{4}{n-2}g_R$$ where $\psi_R$ is the function obtained from Claim~\ref{claim:existence-sol}. By Claim~\ref{claim:existence-sol}, $(M,\tilde g_R)$ is an asymptotically flat in the sense of Definition~\ref{defn:AF}, with vanishing scalar curvature $\mathrm{scal}(\tilde g_R)\equiv 0$. By the positive mass theorem \cites{BiHaoHeShiZhu2026,
BrendleWang2026,SchoenYau1979,
Schoen1989}, $m_{\mathrm{ADM}}(\tilde g_R)\geq 0$. Moreover their mass $m_{\mathrm{ADM}}(\tilde g_R)$ and $m_{\mathrm{ADM}}(g_R)$ are related by 
\begin{equation}\label{eqn:mass-equa}
\begin{split}
&\quad m_{\mathrm{ADM}}(g_R)-m_{\mathrm{ADM}}(\tilde g_R)\\
&=\frac{2}{(n-2)\omega_{n-1}}\cdot\int_M |\nabla^{g_R}\psi_R|^2+c_0 \cdot \mathrm{scal}(g_R)\cdot \psi_R^2\,d\mathrm{vol}_{g_R},
\end{split}
\end{equation}
see \cite{Miao}*{(54)} for example. The claim now follows from \eqref{eqn:mass-equa}, $m_{\mathrm{ADM}}(\tilde g_R)\geq 0$ and $m_{\mathrm{ADM}}(g_R)=0$.
\end{proof}

\medskip

We now claim that the gradient term in Claim~\ref{claim:massss} dominates: 
\begin{claim}\label{claim:esti-from-mass}
We have
\begin{equation}
\lim_{R\to+\infty} \int_{M} |\nabla^{g_R} \psi_R|^2 \,d\mathrm{vol}_{g_R}=0.
\end{equation}
\end{claim}

\begin{proof}[Proof of Claim~\ref{claim:esti-from-mass}]

We will squeeze more non-negativity from the linearisation of scalar  curvature, appeared in \eqref{eqn:energgy}. From the third line of \eqref{eqn:scal-err 1}, \eqref{eqn:F-error} and $\mathrm{scal}(g)\geq 0$, 
\begin{equation}
\begin{split}
\mathrm{scal}(g_R) = {} & \phi_{R}\cdot\mathrm{scal}(g)+\partial_k W^k_R
+\partial_k \phi_R\cdot V^k(g)+(F(g_R) -\phi_R F(g))\\
\geq {} & \partial_k W^k_R +\partial_k \phi_R \cdot V^k(g)  -C_3\e_R R^{4-2n}\chi_{A_{\euc}(0^n,R,2R)}
\end{split}
\end{equation}
so that 
\begin{equation}\label{eqn:scal-ineq-psiR}
\begin{split}
&\quad\int_{M^n} \mathrm{scal}(g_R) \cdot |\psi_R|^2 \,d\mathrm{vol}_{g_R}\\
&\geq \int_{M^n} \left( \partial_k W^k_R +\partial_k \phi_R \cdot V^k(g)  -C_3\e_R R^{4-2n}\chi_{A_{\euc}(0^n,R,2R)}  \right) |\psi_R|^2 \,d\mathrm{vol}_{g_R}\\[2mm]
&=:\mathbf{I}+\mathbf{II}+\mathbf{III}.
\end{split}
\end{equation}
Recall that $W^k_R$ is supported on $A_{\euc}(0^n,R,2R)$. Combining this with Stokes' Theorem and \eqref{eqn:W-error},
\begin{equation}\label{eqn:expres-I}
\begin{split}
\mathbf{I}&=\int_{M^n} \left(\mathrm{div}_{g_R} W_R -\Gamma(g_R)^k_{kl} \cdot W^l_R\right) \cdot |\psi_R|^2 \,d\mathrm{vol}_{g_R}\\
&\geq -\int_{M^n} \langle W_R, \nabla^{g_R} |\psi_R|^2 \rangle \,d\mathrm{vol}_{g_R}-C_9\e_R R^{3-2n} \int_{A_{\euc}(0^n,R,2R)} |\psi_R|^2  \,d\mathrm{vol}_{g_R}\\
&\geq -C_{10}\e_R R^{1-n}\left(\int_{A_{\euc}(0^n,R,2R)} |\psi_R|\cdot|\nabla^{g_R}\psi_R|
+ R^{2-n}  |\psi_R|^2  \,d\mathrm{vol}_{g_R}\right)\\[2mm]
&=:\mathbf{I}_1 +\mathbf{I}_{2}.
\end{split}
\end{equation}
Denote $w_R:=\psi_R-1$, so that \eqref{eqn:Sobolev} implies
\begin{equation}\label{eqn:SOO}
\|w_R\|_{L^{\frac{2n}{n-2}}}\leq \Lambda^{1/2}\cdot \|\nabla^{g_R}w_R\|_{L^{2}}\leq  C'\Lambda^{1/2} \|\nabla^{g_R} \psi_R\|_{L^{2}}=:C'\Lambda^{1/2}\cdot \mathcal{E}_R
\end{equation}
for some uniform $C''>0$. Using $n\geq 4$ and \eqref{eqn:SOO},
\begin{equation}\label{eqn:contoll-I2}
\begin{split}
|\mathbf{I}_2|&\leq C_{10}\e_R R^{3-2n} \int_{A_{\euc}(0^n,R,2R)} |\psi_R|^2  \,d\mathrm{vol}_{g_R}\\
&\leq C_{11}\e_R R^{3-2n} \left(R^n+R^\frac{n+2}2\|w_R\|_{L^\frac{2n}{n-2}} +R^2 \|w_R\|_{L^\frac{2n}{n-2}}^{2}  \right)\\[2mm]
&\leq C_{12}\e_R  \left(R^{3-n}+R^{4-\frac{3n}{2}} \Lambda^{1/2} \mathcal{E}_R +R^{5-2n} \Lambda \mathcal{E}_R^2 \right)\\[2mm]
&\leq 2C_{12}\e_R \Lambda  R^{3-n} \left(1 + \mathcal{E}_R^2 \right)\leq 2C_{12}\e_R  \Lambda \left(1 + \mathcal{E}_R^2 \right)
\end{split}
\end{equation}
and similarly
\begin{equation}\label{eqn:contoll-I1}
\begin{split}
|\mathbf{I}_1|&\leq C_{13} \e_R R^{1-n} \left(\int_{A_{\euc}(0^n,R,2R)} |\psi_R|^2  \,d\mathrm{vol}_{g_R} \right)^{1/2}  \|\nabla^{g_R} \psi_R\|_{L^{2}}\\
&\leq C_{14}\e_R R^{1-n} \left(R^n+R^\frac{n+2}2\|w_R\|_{L^\frac{2n}{n-2}} +R^2 \|w_R\|_{L^\frac{2n}{n-2}}^{2}  \right)^{1/2}\cdot \mathcal{E}_R\\[2mm]
&\leq C_{15} \e_R  \left(R^{1-\frac{n}2}{\mathcal{E}_R}+R^\frac{6-3n}4  \Lambda^{1/2}\mathcal{E}_R^{3/2}  +R^{2-n} \Lambda \mathcal{E}_R^2 \right)\\[2mm]
&\leq 4C_{15}  \Lambda\e_R  \left(1+ \mathcal{E}_R^2 \right).
\end{split}
\end{equation}

The term $\mathbf{III}$ can be controlled  completely identically to that of $\mathbf{I}_2$ except the power on $R$ is increased by $1$:
\begin{equation}\label{eqn:III-need4D}
\begin{split}
|\mathbf{III}|&\leq C_{16} \Lambda R^{4-n}\e_R  \left(1+\mathcal{E}_R^2 \right)\leq C_{16} \Lambda \e_R  \left(1+\mathcal{E}_R^2 \right).
\end{split}
\end{equation}

Finally for $\mathbf{II}$, we use Stokes' Theorem, \eqref{eqn:Sobolev} and decay assumptions to show:
\begin{equation}\label{eqn:estt-II}
\begin{split}
\mathbf{II}
&=\int_{M^n} \partial_k \phi_R \cdot g^{ij}g^{kl}(\partial_j h_{il}-\partial_l h_{ij}) \,|\psi_R|^2 d\mathrm{vol}_{g_{R}}\\
&\geq -C_{17} \e_RR^{-n} \int_{A_{\euc}(0^n,R,2R)} |\psi_R|^2  \,d\mathrm{vol}_{g_R}\\
&\quad -C_{17} \e_R R^{1-n}   \int_{A_{\euc}(0^n,R,2R)} |\psi_R| |\nabla^{g_R}\psi_R|  \,d\mathrm{vol}_{g_R}\\
&\geq -C_{18}\Lambda\e_R (1+\mathcal{E}_R^2)
\end{split}
\end{equation}
where we have used Young's inequality as in the derivation of \eqref{eqn:contoll-I1} and \eqref{eqn:contoll-I2}.

We now substitute \eqref{eqn:expres-I}, \eqref{eqn:contoll-I1}, \eqref{eqn:contoll-I2}, \eqref{eqn:estt-II}, \eqref{eqn:III-need4D} and \eqref{eqn:scal-ineq-psiR} into \eqref{eqn:energgy} to conclude 
\begin{equation}\label{eqn:smallness}
\begin{split}
0 \geq {} & \mathcal{E}_R^2+c_0 (\mathbf{I}+\mathbf{II}+\mathbf{III})\\[2mm]
\geq {} & \mathcal{E}_R^2-C_{19} \e_R\Lambda (1+\mathcal{E}_R^2)\\
\geq {} & \frac12 \mathcal{E}_R^2 -C_{19}\Lambda\e_R
\end{split}
\end{equation}
whenever $R$ is sufficiently large so that $C_{19}\e_R \Lambda \leq \frac12$. This shows that $\mathcal{E}_R=o(1)$ as $R\to+\infty$ and thus completes the proof.
\end{proof}

\medskip
\begin{claim}\label{claim:conformal-factor-limit}
The function $\psi_R$ sub-converges to a smooth function $\psi$ on $M^n$ as $R\to+\infty$ such that on $M^n$, 
$$\nabla^g \psi=0\quad\text{and}\quad \psi=1.$$
\end{claim}
\begin{proof}[Proof of Claim~\ref{claim:conformal-factor-limit}]

Recall that $g_R=g$ on any compact sets as $R\to+\infty$.  By \eqref{eqn:SOO} and Claim~\ref{claim:esti-from-mass}, $\psi_R$ is locally bounded in $L^\frac{2n}{n-2}$. Together with  \eqref{eqn:equationss} on $M^n$,  local interior estimates imply that  $\psi_R$ is locally bounded in $C^k_{\mathrm{loc}}$ for all $k\in \mathbb{N}$. By Arzel\`a–Ascoli theorem, $\psi_R$ sub-converges to some smooth function $\psi$, locally uniformly smoothly. Furthermore, $\nabla^g\psi=0$ follows from Claim~\ref{claim:esti-from-mass}, while $\psi\equiv 1$ follows from \eqref{eqn:SOO}. This completes the proof.
\end{proof}

By \eqref{eqn:equationss} and $\psi\equiv 1$, we have $\mathrm{scal}(g)=c_0^{-1}\psi^{-1}\Delta \psi=0$. This  completes the proof.
\end{proof}

\begin{rem}
The dimension restriction $n\geq4$ appeared in Theorem~\ref{thm:rigidity-ngeq4} is due to the exponent on $R$ appeared in \eqref{eqn:III-need4D}.
\end{rem}

\medskip

With Theorem~\ref{thm:rigidity-ngeq4}, we are now ready to finish the proof of Theorem~\ref{thm:gromov-conj}. 
\begin{proof}[Proof of Theorem~\ref{thm:gromov-conj}]
Let $g_0$ be the given metric and suppose it is non-flat. Let $g(t)$ be the Ricci-DeTurck flow on $M^n\times [0,T]$ constructed in Proposition~\ref{prop:flowing-g_0}, starting from $g_0$ such that $\mathrm{scal}(g(t))\geq 0$ on $[0,T]$. By strong maximum principle, we might assume $\mathrm{scal}(g(t))>0$ for all $t\in (0,T]$ since otherwise $\Ric(g_0)\equiv 0$ so that $g_0$ must be flat by the rigidity of volume comparison which is assumed not to be the case. Fix $t_0\in (0,T]$ and denote $g:=g(t_0)$. By Proposition~\ref{prop:zero-th-decay-preserved} and Proposition~\ref{prop:1st-decay-obtained}, $g$ is non-flat and satisfies the assumption in Theorem~\ref{thm:rigidity-ngeq4}. This is impossible. 
\end{proof}

%
%
%

\appendix

\section{Metrics on \texorpdfstring{$\mathbb{R}^n$}{} with critical decay rate}
\label{sec:append}

In the appendix, we will construct metrics on Euclidean space showing that the decay rate in Theorem~\ref{thm:gromov-conj} is sharp. This should be standard to experts. We include the construction for reader's convenience. More precisely, for every $n\geq 3,c>0$, we will construct a smooth, non-flat and complete metric $g_c$ on $\mathbb{R}^n$ such that
\begin{enumerate}\setlength{\itemsep}{1mm}
\item[(i)] $\mathrm{scal}(g_c)\equiv0$;
\item[(ii)] $ g_c = \left( 1+c|x|^{2-n}+O_2(|x|^{1-n}) \right)g_{\mathrm{euc}}$ as $x\to \infty$,
\end{enumerate}
As observed in \cite{MazurowskiYaoRigidity2026}*{Remark~4}, such a metric can be constructed by perturbing the standard metric on $\mathbb{S}^n$ and then performing a conformal blow-up at one point. We include the details for readers' convenience.

Fix two distinct points $p,q\in\mathbb{S}^n$. By taking a sufficiently small perturbation of $g_{\mathbb{S}^n}$, supported near $q$, one obtains a metric $h$ on $\mathbb{S}^n$ such that
\begin{enumerate}\setlength{\itemsep}{1mm}
\item[(a)] $\mathrm{scal}(h)>0$;
\item[(b)] $h=g_{\mathbb{S}^n}$ near $p$;
\item[(c)] $h$ is not locally conformally flat near $q$.
\end{enumerate}
Let
\[
        L_h := -\frac{4(n-1)}{n-2}\Delta_h+\mathrm{scal}_h
\]
be the conformal Laplacian. By (a), $L_h$ has a positive Green's function $G$ with pole at $p$.

By (b), there are coordinates $y=(y^1,\ldots,y^n)$, centered at $p$, and a positive function $v\in C^\infty(B_1(0))$ such that $h=v^{\frac{4}{n-2}}g_{\mathrm{euc}}$ in $B_{1}(0)$. Set $H:=vG$. Conformal covariance gives
\[
        0 =  L_hG  = v^{-\frac{n+2}{n-2}}L_{g_{\mathrm{euc}}}(vG) = -\frac{4(n-1)}{n-2}  v^{-\frac{n+2}{n-2}}\Delta_{g_{\mathrm{euc}}}H
\]
on $B_1(0)\setminus\{0\}$. Thus $H$ is positive and harmonic there. By B\^{o}cher's theorem,
\[
        H(y)=a|y|^{2-n}+f(y),
\]
where $a>0$ and $f$ is smooth and harmonic on $B_1(0)$. Multiplying $G$ by $a^{-1}$, we may assume that
\[
        H(y)=|y|^{2-n}+f(y).
\]
On $\mathbb{S}^n\setminus\{p\}\cong\mathbb{R}^n$, define
\[
        g:=G^{\frac{4}{n-2}}h.
\]
The conformal transformation law gives
\[
        \mathrm{Scal}_g = G^{-\frac{n+2}{n-2}}L_hG =  0.
\]
To compute the asymptotic expansion, consider the inversion
\[
        \Phi:\mathbb{R}^n\setminus\overline{B_1(0)} \longrightarrow B_1(0)\setminus\{0\}, \qquad \Phi(x)=\frac{x}{|x|^2}.
\]
Since $g=H^{4/(n-2)}g_{\mathrm{euc}}$ near $p$, we have
\begin{align*}
        \Phi^*g
        &= (H\circ\Phi)^{\frac{4}{n-2}} \Phi^*g_{\mathrm{euc}} \\
        &= \left[ |x|^{n-2}  + f\left(\frac{x}{|x|^2}\right) \right]^{\frac{4}{n-2}} |x|^{-4}g_{\mathrm{euc}} \\
        &= \left[  1+ |x|^{2-n}f\left(\frac{x}{|x|^2}\right)  \right]^{\frac{4}{n-2}}g_{\mathrm{euc}} \\
        &= \left[ 1+\frac{4f(0)}{n-2}|x|^{2-n} +O_2(|x|^{1-n})  \right]g_{\mathrm{euc}}.
\end{align*}
In particular, $\frac12g_{\mathrm{euc}}\le\Phi^*g\le2g_{\mathrm{euc}}$
outside a sufficiently large ball. Hence every curve leaving all compact
sets has infinite $g$-length, and $g$ is complete.

With the standard normalization of the ADM mass, the preceding expansion
gives $m_{\mathrm{ADM}}(g)=2f(0)$. Since $\mathbb{R}^n$ is spin, the positive mass theorem \cite{Witten1981} gives $f(0)\ge0$. If $f(0)=0$, rigidity implies that $g$ is flat. Since $h=G^{-4/(n-2)}g$, this would make $h$ locally conformally flat near $q$, contradicting (c). Therefore $f(0)>0$.

Set $c_0:=\frac{4f(0)}{n-2}$. We have obtained a smooth, non-flat and complete scalar-flat metric $g_{c_0}$ satisfying
\[
        g_{c_0} = \left( 1+c_0|x|^{2-n}+O_2(|x|^{1-n})  \right)g_{\mathrm{euc}}.
\]
Next we construct $g_{c}$ for general $c$. For $\lambda>0$, let $D_{\lambda}$ be the dilation transformation:
\[
D_{\lambda}: \mathbb{R}^{n} \to \mathbb{R}^{n}, \ D_{\lambda}(x) = \lambda x.
\]Then the metric $\lambda^{-2}D_{\lambda}^{*}g_{c_{0}}$ satisfies
\[
\lambda^{-2}D_{\lambda}^{*}g_{c_{0}} = (1+c_{0}\lambda^{2-n}|x|^{2-n}+O_{2}(|x|^{1-n}))g_{\mathrm{euc}}.
\]Choosing $\lambda=(c_{0}/c)^{\frac{1}{n-2}}$, we obtain the desired metric $g_{c}$.

\end{document}